\documentclass[11pt, a4paper]{article}
\usepackage[latin1]{inputenc}
\usepackage[T1]{fontenc}
\usepackage[english]{babel}
\usepackage{amssymb, amsmath, amsthm}
\usepackage[colorlinks=true, allcolors=blue]{hyperref}
\usepackage{geometry}
\usepackage{fancyhdr}
\usepackage{relsize}
\usepackage{upgreek}
\usepackage{verbatim}
\usepackage{enumitem}
\usepackage{calrsfs}
\usepackage{todonotes}
\usepackage{tikz}
\usetikzlibrary{matrix,arrows,calc}
\usepackage{package}
\newcommand{\uppertri}{Z_{d,e}}
\newcommand{\uppertripq}{Z_{p,q}}
\DeclareMathOperator{\Exp}{Exp}
\DeclareMathOperator{\prim}{prim}
\makeatletter
\renewcommand*{\@fnsymbol}{\@arabic}
\makeatother
\title{Semi-Stable Chow--Hall Algebras of Quivers and Quantized Donaldson--Thomas Invariants}
\author{H. Franzen%
\thanks{Mathematisches Institut der Universit\"at Bonn, Endenicher Allee 60, 53115 Bonn\newline \href{mailto:franzen@math.uni-bonn.de}{franzen@math.uni-bonn.de}}
\and M. Reineke%
\thanks{Faculty of Mathematics, Ruhr-Universit\"at Bochum, Universit\"atsstra{\ss}e 150, 44780 Bochum\newline \href{mailto:markus.reineke@rub.de}{markus.reineke@rub.de}}%
}

\date{}
\begin{document}

	\maketitle
	\begin{abstract}
		The semi-stable ChowHa of a quiver with stability is defined as an analog of the Cohomological Hall algebra of Kontsevich and Soibelman via convolution in equivariant Chow groups of semi-stable loci in representation varieties of quivers. We prove several structural results on the semi-stable ChowHa, namely isomorphism of the cycle map, a tensor product decomposition, and a tautological presentation. For symmetric quivers, this leads to an identification of their quantized Donaldson--Thomas invariants with the Chow--Betti numbers of moduli spaces.
	\end{abstract}
	
	\section{Introduction}
	
	The Cohomological Hall algebra, or CoHa for short, of a quiver is defined in \cite{KS:11} as an analog of the Hall algebra construction of \cite{Ringel:90} in equivariant cohomology of representation varieties. In \cite{KS:11} the CoHa serves as a tool for the study of quantized Donaldson--Thomas invariants of quivers, and in particular their integrality properties, since it admits a purely algebraic description as a shuffle algebra ``with kernel'' on spaces of symmetric polynomials. In \cite{Efimov:12}, the CoHa of a symmetric quiver is shown to be a free super-commutative algebra, proving the positivity of quantized Donaldson--Thomas invariants in this case.
	
	In another direction, the CoHa is used in \cite{Franzen:13:NCHilb_Loop,Franzen:15:CoHa_Modules} to determine the ring structure on the cohomology of non-commutative Hilbert schemes and more general framed moduli spaces of quiver representations, as defined in \cite{ER:09}.
	
	Already in \cite{Franzen:15:CoHa_Modules} it turns out that a ``local'' version of the CoHa (the semi-stable CoHa), constructed via convolution on semistable loci of representation varieties with respect to a stability, is particularly useful, and that it is also convenient to replace equivariant cohomology by equivariant Chow groups.
	
	In the present paper, we study this local version, called the semi-stable ChowHa, more systematically and demonstrate their utility both for understanding the structure of the CoHa and for the study of quantized Donaldson--Thomas invariants.
	
	We prove the following structural properties of the semi-stable ChowHa:
	
	The equivariant cycle map between the semi-stable ChowHa and the semi-stable CoHa is an isomorphism (Corollary \ref{c:cyclemap}), which can be viewed as a generalization of a result of \cite{KW:95} on the cycle map for fine moduli spaces of quivers. We exhibit a tensor product decomposition (Theorem \ref{t:tensor}) of the CoHa into all semi-stable CoHa's for various slopes of the stability, categorifying the Harder--Narasimhan, or wall-crossing, formula of \cite{Reineke:03}. We give a ``tautological'' presentation of the semi-stable ChowHa (Theorem \ref{taut}), in the spirit of \cite{Franzen:15:Chow_Ring_Quiv}, which generalizes the algebraic description of \cite{KS:11} of the CoHa. Quite surprisingly, such a tautological presentation remains valid for the equivariant Chow groups of stable loci in representation varieties (Theorem \ref{prim}). From this, we conclude that the quantized Donaldson--Thomas invariants of a symmetric quiver are given by the Poincar\'e polynomials of the Chow groups of moduli spaces of stable quiver representations (Theorem \ref{DT=ChowBetti}). This shows that quantized Donaldson--Thomas invariants are of algebro-geometric origin; compare \cite{MR:14} where quantized Donaldson--Thomas invariants are interpreted via intersection cohomology of moduli spaces of semi-stable quiver representations.
	
	The proofs of these structural results basically only use the Harder--Narasimhan stratification of representation varieties of \cite{Reineke:03}, properties of equivariant Chow groups, and the result of Efimov \cite{Efimov:12}. All structural results are illustrated by examples in Section \ref{s:examples}: we first give a complete description of the Hall algebra of a two-cycle quiver, which is the only symmetric quiver with known representation theory apart from the trivial and the one-loop quiver whose CoHa's are already described in \cite{KS:11}. Then we consider the only non-trivial (i.e.\ not isomorphic to the CoHa of a trivial quiver) semi-stable ChowHa for the Kronecker quiver --- we observe that it is not super-commutative, but still has the same Poincar\'e--Hilbert series as a free super-commutative algebra; for this, the representation theory of the Kronecker quiver is used essentially. Then we illustrate the calculation of Chow--Betti numbers of moduli spaces of stable representations in the context of classical invariant theory, and finally hint at an algebraic derivation of the explicit formula of \cite{Reineke:12} for quantized Donaldson--Thomas invariants of multiple loop quivers.

	The paper is organized as follows:
	
	After reviewing basic facts on quiver representations (Section \ref{s:reminder}) and the definition of quantized Donaldson--Thomas invariants (Section \ref{DT}), we define the semi-stable ChowHa in Section \ref{s:sstchowha}. In Section \ref{s:chowhavscoha} we study the cycle map from ChowHa to CoHa, and prove it to be an isomorphism via induction over Harder--Narasimhan strata and framing techniques. The Harder--Narasimhan stratification is also used in Section \ref{s:tensorprod} to derive the tensor product decomposition of the ChowHa. We recall the algebraic description of the CoHa of \cite{KS:11} and Efimov's theorem in Section \ref{s:structurecoha}. Again using the Harder--Narasimhan stratifiction, we obtain the algebraic description of the semi-stable ChowHa in Section \ref{s:taut}. In Section \ref{s:primitive}, the previous results are combined to identify quantized Donaldson--Thomas invariants and Chow--Betti numbers. Finally, the examples mentioned above are developed in Section \ref{s:examples}.
	
	\subsection*{Acknowledgements}
	The authors would like to thank M.\ Brion, B.\ Davison, M.\ Ehrig, and M.\ Young for valuable discussions and remarks. While doing this research, H.F.\ was supported by the DFG SFB / Transregio 45 ``Perioden, Modulr\"aume und Arithmetik algebraischer Variet\"aten''.

	\section{A Reminder on Quiver Representations}\label{s:reminder}
	
	Let $Q$ be a quiver --- i.e.\ a finite oriented graph --- whose set of vertices resp.\ arrows we denote by $Q_0$ resp.\ $Q_1$. We will often suppress the dependency on $Q$ in the notation. The bilinear form $\chi = \chi_Q$ on $\Z^{Q_0}$ defined by
	$$
		\chi(d,e) = \sum_{i \in Q_0} d_ie_e - \sum_{\alpha: i \to j} d_ie_j = \sum_{i,j} (\delta_{i,j} - a_{i,j})d_ie_j
	$$ 
	is called the Euler form of $Q$. Here, $a_{i,j}$ is the number of arrows from $i$ to $j$ in $Q$. We denote the anti-symmetrization $\chi(d,e) - \chi(e,d)$ of the Euler form by $\linspan{d,e}$. Let $\Gamma = \smash{\Z_{\geq 0}^{Q_0}}$ be the monoid of dimension vectors of $Q$. 
	
	Let $k$ be a field. A representation $M$ of $Q$ over $k$ is a collection of finite-dimensional vector spaces $M_i$ with $i \in Q_0$ together with linear maps $M_\alpha: M_i \to M_j$ for every arrow $\alpha: i \to j$. See \cite{ASS:06} for more details. The tuple $\dimvect M = (\dim M_i \mid i \in Q_0)\in\Gamma$ is called the dimension vector of $M$. For a dimension vector $d\in\Gamma$, we define $R_d(k)$ to be the vector space
	$$
		R_d(k) = \bigoplus_{\alpha: i \to j} \Hom(k^{d_i},k^{d_j})
	$$
	on which we have an action of the group $G_d(k) = \prod_{i \in Q_0} \Gl_{d_i}(k)$ via base change. An element of $R_d(k)$ is a representation of $Q$ on the vector spaces $\smash{(k^{d_i})}_i$. Being an affine space, $R_d(k)$ admits a $\Z$-model, i.e.\ there exists a scheme $R_d$ whose set of $k$-valued points is $R_d(k)$. Likewise, there is a group scheme $G_d$ which is a $\Z$-model for $G_d(k)$.
	
	We introduce a stability condition $\theta$ of $Q$, that is, a linear form $\Z^{Q_0} \to \Z$. For a non-zero dimension vector $d$, the rational number
	$$
		\frac{\theta(d)}{\sum_i d_i}
	$$
	is called the $\theta$-slope of $d$. For a rational number $\mu$, let $\Gamma^{\theta,\mu}$ be the submonoid of all $d \in \Gamma$ with $d = 0$ or whose $\theta$-slope is $\mu$. If $M$ is a non-zero representation of $M$ of $Q$ over $k$, the $\theta$-slope of $M$ is defined as the slope of its dimension vector.
	A representation $M$ of $Q$ over $k$ is called $\theta$-\textbf{semi-stable} if no non-zero subrepresentation of $M$ has larger $\theta$-slope than $M$. It is called $\theta$-\textbf{stable} if the $\theta$-slope of every non-zero subrepresentation $M'$ is strictly less than the slope of $M$, unless $M'$ agrees with $M$. There is a Zariski-open subset $\smash{R_d^{\theta-\sst}}$ of the scheme $R_d$ whose set of $k$-valued points is the set of $\theta$-semi-stable representations of $Q$. There is also an open subset $\smash{R_d^{\theta-\st}}$ of $\smash{R_d^{\theta-\sst}}$ parametrizing absolutely $\theta$-stable representations, that means $\smash{R_d^{\theta-\st}}(k)$ consists of those $M \in R_d(k)$ such that $M \otimes_k K$ is $\theta$-stable for every finite extension $K \mid k$.
	
	\section{Quantum Donaldson--Thomas Invariants} \label{DT}
	
	Fix a prime power $q$ and let $\F = \F_q$ be the finite field with $q$ elements. Then $R_d(\F)$ and $G_d(\F)$ are finite sets, and we consider the set of orbits $R_d(\F)/G_d(\F)$. We define the \textbf{completed Hall algebra} of $Q$ as the vector space
	$$
		H_q = H_q((Q)) = \{ f \mid f: \bigsqcup_{d \in \Gamma} R_d(\F)/G_d(\F) \to \Q \}
	$$
	equipped with the following convolution-type multiplication: For two functions $f$ and $g$, we define
	$$
		(f*g)(X) = \sum_{U \sub X} f(U) g(X/U),
	$$
	the sum ranging over all subrepresentations of $X$. Note that this sum is finite. This multiplication turns $H_q$ into an associative algebra. We define yet another algebra. Set $\T_q := \Q(q^{1/2})[[t_i \mid i \in \Q_0]]$ and let the multiplication be given by
	$$
		t^d \circ t^e = (-q^{1/2})^{\linspan{d,e}} t^{d+e}.
	$$
	It is shown in \cite{Reineke:03} that the so-called integration map $\int: H_q \to \T_q$ defined by
	$$
		\int f = \sum_{[X]} \frac{(-q^{1/2})^{\chi(\dimvect X, \dimvect X)}}{\sharp \Aut(X)} \cdot f(X) \cdot t^{\dimvect X}
	$$
	is a homomorphism of algebras. Define $\one \in H_q$ to be the function with $\one(X) = 1$ for all $[X]$. An easy computation shows that $A(q,t) := \int \one$ equals
	$$
		A(q,t) = \sum_d (-q^{1/2})^{-\chi(d,d)} \prod_i \prod_{\nu = 1}^{d_i} (1-q^{-\nu})^{-1} t^d.
	$$
	
	For a stability condition $\theta$ of $Q$ and a rational number $\mu$, we define $\one^{\theta,\mu} \in H_q$ as the sum of the characteristic functions on $\smash{R_d^{\theta-\sst}}(\F)/G_d(\F)$ over all $d \in \Gamma^{\theta,\mu}$. Set $A^{\theta,\mu}(q,t) = \int \one^{\theta,\mu}$. Using a Harder--Narasimhan type recursion, it is shown in \cite{Reineke:03} that
	
	\begin{thm}[\cite{Reineke:03}] \label{wcf}
		In $H_q$, we have $\one = {\prod\limits^\ot}_{\mu \in \Q} \one^{\theta,\mu}$.
	\end{thm}
	
	This implies that the the series $A$ and $A^{\theta,\mu}$ relate in the same way in the twisted power series ring $\T_q$.
	
	Let $R$ be the power series ring $\Q(q^{1/2})[[t_i \mid i \in Q_0]]$ with the usual multiplication. Let $R_+$ be the set of power series without constant coefficient. There exists a unique bijection $\Exp: R_+ \to 1 + R_+$ such that $\Exp(f + g) = \Exp(f)\Exp(g)$ and
	$$
		\Exp(q^{k/2}t^d) = \frac{1}{1-q^{k/2}t^d}
	$$
	for every $k \in \Z$ and $d \in \Gamma$. This function is called the plethystic exponential. We call the stability condition $\theta$ generic for the slope $\mu\in \Q$ (or $\mu$-generic) if $\linspan{d,e} = 0$ for all $d,e \in \Gamma^{\theta,\mu}$. Assuming that $\theta$ is $\mu$-generic, the series $A^{\theta,\mu}$ can be displayed as a plethystic exponential:
	
	\begin{thm}[{\cite{KS:08}}] \label{dt_thm}
		For a $\mu$-generic stability condition $\theta$, there are polynomials $\tilde{\Omega}_d^\theta(q) = \sum_k \tilde{\Omega}_{d,2k}^\theta q^k$ in $\Z[q]$ for every non-zero dimension vector $d$ of slope $\mu$ such that
		$$
			A^{\theta,\mu}(q^{-1},t) = \Exp\left( \frac{1}{1-q}\sum_d  (-q^{1/2})^{\chi(d,d)} \tilde{\Omega}_d^\theta(q)t^d \right).
		$$
	\end{thm}
	
	\begin{defn*}
		If $\theta$ is a $\mu$-generic stability condition and $d$ is a non-zero dimension vector of slope $\mu$ then the coefficients of
		$$
			\Omega_d^{\theta}(q) = \sum_{k \in \Z} \Omega_{d,k}^\theta q^{k/2} := q^\frac{\chi(d,d)}{2}\tilde{\Omega}_d^\theta(q) \in \Z[q^{\pm 1/2}]
		$$
		are called the \textbf{quantum Donaldson--Thomas invariants} of $Q$ with respect to $\theta$.
	\end{defn*}
	
	When the quiver $Q$ is symmetric, that means the Euler form of $Q$ is a symmetric bilinear form, then every stability condition is $\mu$-generic for every value $\mu \in \Q$. For example, the trivial stability condition is $0$-generic and we can define the Donaldson--Thomas invariants $\smash{\Omega_{d,k}^0}$ for every $d \neq 0$. Note that Theorem \ref{wcf} implies $\smash{\Omega_{d,k}^0} = \smash{\Omega_{d,k}^\theta}$ for every $\theta$. We may therefore write $\Omega_{d,k}$ in this case.

	\section{The Semi-Stable ChowHa}\label{s:sstchowha}
	
	Fix an algebraically closed field $k$. Abusing the notation from the first section, we will use the symbols $R_d$, $\smash{R_d^{\theta-\sst}}$, $\smash{R_d^{\theta-\st}}$, and $G_d$ for the base extensions of the respective $\Z$-models to $\Spec k$.
	
	Let $d$ be a dimension vector for $Q$. We define $\AA_d^{\theta-\sst}$ to be the $G_d$-equivariant Chow ring with rational coefficients of the semi-stable locus
	$$
		\AA_d^{\theta-\sst}(Q) = A_{G_d}^*(R_d^{\theta-\sst})_\Q.
	$$
	For the definition of equivariant Chow groups/rings, see Edidin--Graham's article \cite{EG:98}. As we will always work with rational coefficients, we will often omit it in the notation. We define $\AA^{\theta-\sst,\mu}$ as the graded vector space
	$$
		\AA^{\theta-\sst,\mu}(Q) = \bigoplus_{d \in \Gamma^{\theta,\mu}} \AA_d^{\theta-\sst}
	$$
	We mimic Kontsevich--Soibelman's construction of the Cohomological Hall algebra (CoHa) of a quiver with stability and trivial potential from \cite{KS:11}.
	
	For two dimension vectors $d$ and $e$ of the same slope $\mu$, we set $\uppertri$ as the subspace of $R_{d+e}$ of representations $M$ which have a block upper triangular structure as indicated, i.e.\ for every arrow $\alpha: i \to j$, the linear map $M_\alpha$ sends the first $d_i$ coordinate vectors of $k^{d_i+e_i}$ into the subspace of $k^{d_j+e_j}$ spanned by the first $d_j$ coordinate vectors. We consider
	$$
		R_d \times R_e \ot \uppertri \to R_{d+e},
	$$
	the left hand map sending a representation $M = \big(\smash{\begin{smallmatrix} M' & * \\ & M'' \end{smallmatrix}}\big)$ to the pair $(M',M'')$, and the right hand map being the inclusion. These maps are called Hecke correspondences. For a short exact sequence $0 \to M' \to M \to M'' \to 0$ of representations \emph{of the same slope}, $M$ is $\theta$-semi-stable if and only if both $M'$ and $M''$ are. We thus obtain cartesian squares
	\begin{center}
		\begin{tikzpicture}[description/.style={fill=white,inner sep=2pt}]
			\matrix(m)[matrix of math nodes, row sep=.2em, column sep=2em, text height=1.5ex, text depth=0.25ex]
			{
				R_d \times R_e & \uppertri & R_{d+e} \\
				\rotatebox{90}{$\sub$} & \rotatebox{90}{$\sub$} & \rotatebox{90}{$\sub$} \\
				R_d^{\sst} \times R_e^{\sst} & \uppertri \cap R_{d+e}^{\sst} & R_{d+e}^{\sst}. \\
			};
			\path[->, font=\scriptsize]
			(m-1-2) edge (m-1-1)
			(m-1-2) edge (m-1-3)
			(m-3-2) edge (m-3-1)
			(m-3-2) edge (m-3-3)
			;
		\end{tikzpicture}
	\end{center}
	The action of $G = G_{d+e}$ on $R_{d+e}$ restricts to an action of the parabolic $P = \big(\smash{\begin{smallmatrix} G_d & * \\ & G_e \end{smallmatrix}}\big)$ on $\smash{\uppertri}$, and the map $\smash{\uppertri} \to R_d \times R_e$ is compatible with the action of its Levi $L = G_d \times G_e$ on $R_d \times R_e$. With respect to these actions, the respective semi-stable loci are invariant. This gives rise to morphisms
	$$
		(R_d^{\sst} \times R_e^{\sst}) \times^L G \ot \uppertri^{\sst} \times^L G \to \uppertri^{\sst} \times^P G \to R_{d+e}^{\sst} \times^P G \to R_{d+e}.
	$$
	We see that
	\begin{itemize}
		\item $(R_d^{\sst} \times R_e^{\sst}) \times^L G \ot \uppertri^{\sst} \times^L G$ is a $G$-equivariant (trivial) vector bundle,
		\item $\uppertri^{\sst} \times^L G \to \uppertri^{\sst} \times^P G$ is a fibration whose fiber $P/L$ is an affine space (as $L$ is the Levi of $P$ in $G$) and thus induces an isomorphism in $G$-equivariant intersection theory,
		\item $\uppertri^{\sst} \times^P G \to R_{d+e}^{\sst} \times^P G$ is the zero section of a $G$-equivariant vector bundle whose rank is $s_1 = \sum_{\alpha: i \to j} d_i e_j$, and
		\item $R_{d+e}^{\sst} \times^P G \to R_{d+e}$ is proper as $G/P$ is complete, and the dimension of $G/P$ is $s_0 = \sum_i d_ie_i$.
	\end{itemize}
	The above morphisms give rise to maps in equivariant intersection theory \label{coha_mult}
	$$
		A_L^n(R_d^{\sst} \times R_e^{\sst}) \xto{}{\cong} A_L^n \big( \uppertri^{\sst} \big) \xot{}{\cong} A_P^n \big( \uppertri^{\sst} \big) \to A_P^{n+s_1}(R_{d+e}^{\sst}) \to A_G^{n+s_1-s_0}(R_{d+e}^{\sst}).
	$$
	Note that $s_1-s_0$ precisely equals $-\chi(d,e)$, the negative of the Euler form of $d$ and $e$. Composing with the equivariant exterior product map $A_{G_d}^*(R_d^{\sst}) \otimes A_{G_e}^*(R_e^{\sst}) \to A_L^*(R_d^{\sst} \times R_e^{\sst})$, we obtain a linear map
	$$
		\AA_d^{\sst} \otimes \AA_e^{\sst} \to \AA_{d+e}^{\sst}.
	$$
	The proof of \cite[Thm.\ 1]{KS:11} also shows that we thus obtain an associative $\Gamma^{\theta,\mu}$-graded algebra. In analogy to Kontsevich--Soibelman's terminology, we define:
	
	\begin{defn*}
		The algebra $\AA^{\theta-\sst,\mu}(Q)$ is called the $\theta$-semi-stable \textbf{Chow--Hall algebra} (ChowHa) of slope $\mu$ of $Q$.
	\end{defn*}
	
	For the special case that $\theta$ is zero (i.e.\ $\smash{R_d^{\sst}} = R_d$) and $\mu = 0$, we write $\AA$ instead of $\AA^{0-\sst,0}$ and call it the \textbf{ChowHa} of $Q$.
	
	\section{ChowHa vs.\ CoHa}\label{s:chowhavscoha}
	
	We discuss the relation between the semi-stable ChowHa and Kontsevich--Soibelman's semi-stable CoHa. Let $k$ be the field of complex numbers. There is an equivariant analog (see \cite[2.8]{EG:98}) of the cycle map from \cite[Chap.\ 19]{Fulton:98}. Concretely, there is a homomorphism of rings which doubles degrees
	$$
		A_{G_d}^*(\smash{R_d^{\theta-\sst}}) \to H_{G_d}^*(\smash{R_d^{\theta-\sst}})
	$$
	for every stability condition $\theta$ and every dimension vector $d$.
	
	\begin{thm} \label{cyc}
		The equivariant cycle map $A_{G_d}^*(\smash{R_d^{\theta-\sst}}) \to H_{G_d}^*(\smash{R_d^{\theta-\sst}})$ is an isomorphism. In particular, $\smash{R_d^{\theta-\sst}}$ has no odd-dimensional $G_d$-equivariant cohomology.
	\end{thm}

	This theorem generalizes a result due to King--Walter (cf.\ \cite[Thm.\ 3 (c)]{KW:95}). They show the above assertion for an acyclic quiver, an indivisible dimension vector and a stability condition for which stability and semi-stability coincide. In this case, there exists a geometric $PG_d$-quotient $\smash{R_d^{\theta-\mathrm{(s)st}}} \to \smash{M_d^{\theta}}$, and the $G_d$-equivariant Chow/cohomology groups agree with the tensor product of the ordinary Chow/cohomology groups of the quotient with a polynomial ring $\Q[z] \cong A_{\G_m}^*(\pt) \cong H_{\G_m}^*(\pt)$.
	
	We need a general lemma in order to prove Theorem \ref{cyc}. Let $X$ be a complex algebraic scheme embedded into a non-singular variety $\bar{X}$ of complex dimension $N$; for the rest of this section, a \emph{scheme} will be a complex algebraic scheme (cf.\ \cite[B.1.1]{Fulton:98}) which admits such an embedding. The Borel--Moore homology $H_k(X)$ is isomorphic to the singular cohomology $H^{2N-k}(\bar{X},\bar{X}-X)$. We consider the cycle map $\cl: A_k(X) \to H_{2k}(X)$.
	
	\begin{lem} \label{3space}
		Let $X$ be a scheme, $Y$ a closed subscheme of $X$ and $U$ the open complement.
		\begin{enumerate}
			\item If both $Y$ and $U$ have no odd-dimensional homology then $X$ does not have odd-dimensional homology.
			\item If $A_k(Y) \to H_{2k}(Y)$ and $A_k(U) \to H_{2k}(U)$ are isomorphisms and $H_{2k+1}(U) = 0$ then $A_k(X) \to H_{2k}(X)$ is an isomorphism.
			\item Suppose that $A_k(Y) \to H_{2k}(Y)$ and $A_k(X) \to H_{2k}(X)$ are isomorphisms and $H_{2k-1}(Y) = 0$. Then $A_k(U) \to H_{2k}(U)$ is also an isomorphism.
		\end{enumerate}
	\end{lem}
	
	\begin{proof}
		The first assertion is clear by the long exact sequence in homology. To prove the second statement, we consider the diagram
		\begin{center}
			\begin{tikzpicture}[description/.style={fill=white,inner sep=2pt}]
				\matrix(m)[matrix of math nodes, row sep=1.5em, column sep=1.5em, text height=1.5ex, text depth=0.25ex]
				{
					& & A_k(Y) & A_k(X) & A_k(U) & 0 & \\
					\ldots & H_{2k+1}(U) & H_{2k}(Y) & H_{2k}(X) & H_{2k}(U) & H_{2k-1}(Y) & \ldots \\
				};
				\path[->, font=\scriptsize]
				(m-1-3) edge 	(m-1-4)
				(m-1-4) edge 	(m-1-5)
				(m-1-5) edge 	(m-1-6)
				(m-2-1) edge 	(m-2-2)
				(m-2-2) edge 	(m-2-3)
				(m-2-3) edge 	(m-2-4)
				(m-2-4) edge	(m-2-5)
				(m-2-5) edge	(m-2-6)
				(m-2-6) edge	(m-2-7)
				(m-1-3) edge	(m-2-3)
				(m-1-4) edge	(m-2-4)
				(m-1-5) edge	(m-2-5)
				;
			\end{tikzpicture}
		\end{center}		
		The left and right vertical maps being isomorphisms and the leftmost term in the lower row being zero by assumption, the claim follows by applying the snake lemma. The third claim follows by diagram chase in the same diagram. We give the proof for completeness. Let $u \in H_{2k}(U)$. There exists $x \in H_{2k}(X)$ with $j^*x = u$ where $j: U \to X$ is the open embedding. We find a unique $\xi \in A_k(X)$ with $\cl_X\xi = x$, so $u = j^*\cl_X\xi = \cl_Uj^*\xi$ which proves the surjectivity of $\cl_U$. To show that $\cl_U$ is injective, let $\upsilon \in A_k(U)$ with $\cl_U\upsilon = 0$. For an inverse image $\xi \in A_k(X)$ of $\upsilon$ under $j^*$, we obtain $j^*\cl_X\xi = 0$, whence there exists $y \in H_{2k}(Y)$ such that $i_*y = \cl_X\xi$. Here $i: Y \to X$ denotes the closed immersion. Let $\eta \in A_k(Y)$ be the unique cycle with $\cl_Y\eta = y$. We get $\cl_Xi_*\eta = i_*y = \cl_X \xi$ and thus $i_*\eta = \xi$ by injectivity of $\cl_X$. This implies $\upsilon = j^*i_*\eta = 0$.
	\end{proof}
	
	An immediate consequence of the above lemma is the following
	
	\begin{lem} \label{filt}
		Suppose that a scheme $X$ has a filtration $X = X_N \supseteq \ldots \supseteq X_1 \supseteq X_0 = \emptyset$ by closed subschemes such that the cycle map for the successive complements $S_n = X_n - X_{n-1}$ is an isomorphism for all $n$. Then $\cl_X$ is an isomorphism and, moreover, we have
		$$
			A_*(X)\cong\bigoplus_{i=0}^NA_*(S_i)\mbox{ and } A^*(X)\cong\bigoplus_{i=0}^NA^{*-{\codim}_XS_i}(S_i).$$
	\end{lem}
	
	We now turn to an equivariant setup. Let $G$ be a reductive linear algebraic group acting on a scheme $X$ of complex dimension $n$. For an index $i$, we choose a representation $V$ of $G$ and an open subset $E \sub V$ such that a principal bundle quotient $E/G$ exists and such that $\codim_V(V-E) > n-i$. Then, the group 
	$$
		A_k^G(X) = A_{i+\dim V-\dim G}(X \times^G E)
	$$
	is independent of the choice of $E$ and $V$. In the same vein, for an index $j$ with $2\codim_V(V-E) > 2n-j$, we can define equivariant Borel--Moore homology via ordinary Borel--Moore homology, namely
	$$
		H_j^G(X) = H_{j+2\dim V-2\dim G}(X \times^G E)
	$$
	(see \cite{EG:98}). If $X$ is smooth then $\smash{H_j^G(X)}$ is dual to $H^{2n-j}(X \times^G E)$ which is isomorphic to $H^{2n-j}(X \times^G EG) = H_G^{2n-j}(X)$ (where $EG$ is the classifying space for $G$). 
	
	We consider the equivariant cycle map $\cl: \smash{A_k^G(X)} \to \smash{H_{2k}^G(X)}$ which is defined as the ordinary cycle map $\cl: \smash{A_{i+\dim V-\dim G}(X \times^G E)} \to \smash{H_{2k+2\dim V-2\dim G}(X \times^G E)}$ (again independent of $E \sub V$). For complementary open/closed subschemes $U$ and $Y$ of $X$ which are $G$-invariant, we choose $E \sub V$ such that the principal bundle quotient $E/G$ exists and $\codim_V(V-E) > n-i$ (note that all the equivariant versions of the groups appearing in the diagram in the proof of Lemma \ref{3space} can be defined using $E$) and apply Lemma \ref{3space} to the complementary open/closed subschemes $U \times^G E$ and $Y \times^G E$. We thus obtain
	
	\begin{cor}\label{cor43} 
		In the above equivariant situation, Lemmas \ref{3space} and \ref{filt} hold for equivariant Chow/ Borel--Moore homology groups.
	\end{cor}

	\begin{proof}[\textit{Proof of Theorem \ref{cyc}}]
		We prove Theorem \ref{cyc} in two steps:
		\begin{enumerate}
			\item Prove that the odd-dimensional equivariant cohomology of the $\theta$-semi-stable locus vanishes by reducing the arbitrary case to a situation where stability and semi-stability agree. There, the statement is known thanks to \cite{Reineke:03}.
			\item Prove that the equivariant cycle map $A_k^{G_d}(R_d^{\theta-\sst}) \to H_{2k}^{G_d}(R_d^{\theta-\sst})$ is an isomorphism by induction over the Harder--Narasimhan strata.
		\end{enumerate}
		
		\textbf{Step 1:} First, assume that $d$ is a $\theta$-coprime dimension vector. This means that there is no sub-dimension vector $0 \neq d' \leq d$ with the same slope as $d$ apart from $d$ itself. In this case, $\theta$-semi-stability and $\theta$-stability on $R_d$ agree and there exists a smooth geometric $PG_d$-quotient $R_d^{\theta-\mathrm{(s)st}} \to M_d^{\theta}$. Here, $PG_d = G_d/\C^\times$. In \cite[Thm.\ 6.7]{Reineke:03} it is shown that the odd-dimensional cohomology of $M_d^{\theta}$ vanishes. But as by the existence of a geometric quotient
		$$
			H_{G_d}^*(R_d^{\theta-\sst}) \cong H^*(M_d^{\theta}) \otimes H_{\C^\times}^*(\pt),
		$$
		it follows that $R_d^{\theta-\sst}$ has no odd-dimensional cohomology.
		
		Now, let $d$ be arbitrary. We show that for a fixed --- not necessarily positive --- integer $k$, there exist an acyclic quiver $\smash{\hat{Q}}$, a stability condition $\smash{\hat{\theta}}$, a $\smash{\hat{\theta}}$-coprime dimension vector $\smash{\hat{d}}$, and an integer $s \geq 0$ (all depending on $k$) for which
		\begin{align}
			H_{k}^{G_d}(R_d^{\theta-\sst}) &\cong H_{k+2s}^{PG_{\hat{d}}}(R_{\hat{d}}^{\hat{\theta}-\mathrm{(s)st}}(\hat{Q})).
		\end{align}
		For a dimension vector $n$ of $Q$, we consider $\smash{\hat{R}}_{d,n} = R_d \times F_n$ where $F_n = \bigoplus_i \Hom(\C^{n_i},\C^{d_i})$. The space $\smash{\hat{R}}_{d,n}$ is the space of representations of the framed quiver $\smash{\hat{Q}}$ of dimension vector $\smash{\hat{d}}$ which arise as follows (cf.\ \cite[Def.\ 3.1]{ER:09}): we add an extra vertex $\infty$ to the vertexes of $Q$, i.e.\ $\smash{\hat{Q}}_0 = Q_0 \sqcup \{\infty\}$ and, in addition to the arrows of $Q$, we have $n_i$ arrows from $\infty$ heading to $i$ for all $i \in Q_0$. The dimension vector $\smash{\hat{d}}$ is defined by $\smash{\hat{d}}_i = d_i$ for $i \in Q_0$ and $\smash{\hat{d}}_\infty = 1$ and is indivisible. The structure group $\smash{G_{\hat{d}}}$ is $\C^\times \times G_d$, whence we can identify $\smash{PG_{\hat{d}}}$ with $G_d$. We define $\smash{\hat{\theta}}$ in the same way as in \cite[Def.\ 3.1]{ER:09}. The following are equivalent for a framed representation $(M,f) \in \smash{\hat{R}_{d,n}}$ (cf.\ \cite[Prop.\ 3.3]{ER:09}):
		\begin{itemize}
			\item $(M,f)$ is $\smash{\hat{\theta}}$-semi-stable,
			\item $(M,f)$ is $\smash{\hat{\theta}}$-stable,
			\item $M$ is $\theta$-semi-stable and the ($\theta$-)slope of every proper subrepresentation $M'$ of $M$ which contains the image of $f$ is strictly less than the slope of $M$.
		\end{itemize}
		We denote the set of $\smash{\hat{\theta}}$-(semi-)stable points of $\smash{\hat{R}_{d,n}}$ with $\smash{\hat{R}_{d,n}^\theta}$. It is, by the above characterization, an open subset of $\smash{R_d^{\theta-\sst}} \times F_n$. Let $\smash{\hat{R}_{d,n}^x}$ denote the complement of $\smash{\hat{R}_{d,n}^\theta}$ inside $\smash{R_d^{\theta-\sst}} \times F_n$. As $\smash{R_d^{\theta-\sst}} \times F_n$ is a $G_d$-equivariant vector bundle over $R_d^{\theta-\sst}$, we obtain
		$$
			H_k^{G_d}(R_d^{\theta-\sst}) \cong H_{k+2d\cdot n}^{G_d}(R_d^{\theta-\sst} \times F_n)
		$$
		where $d\cdot n := \sum_i d_i n_i = \dim_\C F_n$. We thus obtain a long exact sequence
		\begin{align*}
			\ldots &\to H_{k+2d\cdot n}^{G_d}(\hat{R}_{d,n}^x) \to H_{k}^{G_d}(R_d^{\theta-\sst}) \to H_{k+2d\cdot n}^{G_d}(\hat{R}_{d,n}^\theta) \to H_{k-1+2d\cdot n}^{G_d}(\hat{R}_{d,n}^x) \to \ldots
		\end{align*}
		in equivariant Borel--Moore homology. The equivariant BM homology groups $H_l^{G_d}(\hat{R}_{d,n}^x)$ vanish if $l$ exceeds $2 \dim \smash{\hat{R}}_{d,n}^x$. So in order to show that (1) is an isomorphism, it suffices to find a framing datum $n$ such that the (complex) dimension of $\smash{\hat{R}_{d,n}^x}$ is smaller than $(k-1)/2+d\cdot n$.
		As shown in the proof of \cite[Thm.\ 3.2]{Franzen:15:CoHa_Modules}, $\smash{\hat{R}_{d,n}^x}$ is the union of Harder--Narasimhan strata
		$$
			\hat{R}_{d,n}^x = \bigsqcup \hat{R}_{(\hat{p},q),n}^{\HN}
		$$
		over all proper sub-dimension vectors $p$ of $d$ which have the same slope (and $q = d - p$).
		The set $\smash{\hat{R}_{(\hat{p},q),n}^{\HN}}$ is defined as follows: let $L(M,f)$ be minimal among those representations of the same slope as $M$ which contain $\im f$. We set $\smash{\hat{R}_{(\hat{p},q),n}^{\HN}}$ as the set of all $(M,f) \in R_d^{\theta-\sst} \times F_n$ with $\dimvect L(M,f) = p$. As
		$$
			\hat{R}_{(\hat{p},q),n}^{\HN} \cong \Big( \big( \begin{smallmatrix} R_p^{\sst} & * \\ & R^{\sst}_q \end{smallmatrix} \big) \times \big( \begin{smallmatrix} F_p \\ 0 \end{smallmatrix} \big) \Big) \times^{P_{p,q}} G_d,
		$$
		the dimension of this stratum --- if non-empty --- equals
		\begin{align*}
			\sum_{\alpha: i \to j} (d_id_j - p_iq_j) + \sum_i p_in_i - \sum_i (d_i^2 - p_iq_i) + \sum_i d_i^2 
			&= \dim (R_d) + d \cdot n + \chi(p,q) - q \cdot n.
		\end{align*}
		Choosing $n$ large enough such that
		$$
			q \cdot n > \dim(R_d) - \frac{k-1}{2} + \chi(d-q,q)
		$$
		for all sub-dimension vectors $0 \neq q \leq d$ of the same slope as $d$ (which is possible as these are finitely many non-zero dimension vectors $q$), we find that the dimension of $\smash{\hat{R}_{d,n}^x}$ is smaller than $(k-1)/2 + d \cdot n$, as desired.
		
		Similar arguments were also used by Davison--Meinhardt in \cite[Le. 4.1]{DM:16}.

		\textbf{Step 2:} Let $Q$, $\theta$ and $d$ be arbitrary. We consider the open/closed complementary subsets $R_d^{\sst}$ and $R_d^{\unst}$. As $R_d^{\sst}$ is smooth (of dimension $n = \sum_{\alpha:i \to j} d_id_j$), we have $\smash{A_i^{G_d}(R_d^{\sst})} \cong \smash{A_{G_d}^{n-i}(R_d^{\sst})}$ and $\smash{H_j^{G_d}(R_d^{\sst})} \cong \smash{H_{G_d}^{2n-j}(R_d^{\sst})}$. By Corollary \ref{cor43} (concretely, the equivariant analog of part (3) of Lemma \ref{3space}), it suffices to show that 
		$$
			\cl: A_*^{G_d}(R_d^{\unst}) \to H_{2*}^{G_d}(R_d^{\unst})
		$$ 
		is an isomorphism. If $R_d^{\unst} = \emptyset$, then the assertion is clear. So let us assume that $R_d^{\unst}$ is non-empty. The unstable locus admits a stratification into locally closed (irreducible) subsets $R_{d^*}^{\HN}$, the Harder--Narasimhan strata, by \cite[Prop.\ 3.4]{Reineke:03}. By \cite[Prop.\ 3.7]{Reineke:03}, they can be ordered in such a way that the union of the first $n$ strata is closed for all $n$, thus yielding a filtration by $G_d$-invariant closed subsets like in Lemma \ref{filt}. Thus it suffices to prove that
		$$
			A_*^{G_d}(R_{d^*}^{\HN}) \to H_*^{G_d}(R_{d^*}^{\HN})
		$$
		is an isomorphism for all HN types $d^* = (d^1,\ldots,d^l)$ of $d$; this includes showing that all HN strata have even cohomology. But by the proof of \cite[Prop.\ 3.4]{Reineke:03}, we have 
		$$
			R_{d^*}^{\HN} \cong Z_{d^*}\times^{P_{d^*}} G_d,
		$$
		where $Z_{d^*}$ is a (trivial) vector bundle over $R_{d^1}^{\sst} \times \ldots \times R_{d^l}^{\sst}$, and $P_{d^*}$ is a parabolic subgroup of $G_d$ with Levi $G_{d^1}\times \ldots \times G_{d^l}$.
		
		In particular, $R_{d^*}^{\HN}$ is smooth, therefore we can again identify Borel--Moore homology with cohomology. Moreover, $\smash{A_{G_d}^i(R_{d^*}^{\HN}) }\cong \smash{A_{G_{d^1} \times \ldots \times G_{d^l}}^i(R_{d^1}^{\sst} \times \ldots \times R_{d^l}^{\sst})}$, and similarly for equivariant cohomology. This shows in particular that the equivariant odd-dimensional cohomology of $R_{d^*}^{\HN}$ vanishes by the first step of the proof. Now we argue by induction on the dimension vector $d$, where the set of dimension vectors is partially ordered by $d \leq e$ if $d_i \leq e_i$ for all $i$; with respect to this order, all $d^\nu$'s are strictly smaller than $d$. We thus assume that the equivariant cycle map for each $R_{d^\nu}^{\sst}$ is an isomorphism. Then, \cite[Le.\ 6.2]{Totaro:99}, which can be generalized to equivariant Chow groups, implies that the equivariant exterior product map
		$$
			A_{G_{d^1}}^*(R_{d^1}^{\sst}) \otimes \ldots \otimes A_{G_{d^l}}^*(R_{d^l}^{\sst}) \to A_{G_{d^1} \times \ldots \times G_{d^l}}^*(R_{d^1}^{\sst} \times \ldots \times R_{d^l}^{\sst})
		$$
		is an isomorphism (even with integral coefficients). As the $R_{d^\nu}^{\sst}$'s have even cohomology, the K\"unneth map is an isomorphism. We are thus reduced to proving the assertion for minimal dimension vectors $d$, i.e.\ $d= 0$. But there the statement is obviously true.
	\end{proof}

	\begin{rem}
		Theorem \ref{cyc} is valid for integer coefficients.
	\end{rem}

	\begin{cor}\label{c:cyclemap}
		The cycle map induces an isomorphism $\AA^{\sst,\mu} \to \HH^{\sst,\mu}$ of algebras.
	\end{cor}
	
	\begin{proof} 
		The multiplication both in the semi-stable ChowHa and in the semi-stable CoHa $\HH^{\sst,\mu}$ are constructed by means of the same Hecke correspondences. Moreover, the cycle map is compatible with push-forward and pull-back.
	\end{proof}
	
	\section{Tensor Product Decomposition}\label{s:tensorprod}
	
	We apply Corollary \ref{cor43} to the Harder--Narasimhan filtration, like in the proof of Theorem \ref{cyc}. We obtain that $\smash{A_*^{G_d}}(R_d) \cong \bigoplus_{d^*} \smash{A_*^{G_d}}(R_{d^*}^{\HN})$, where the sum ranges over all Harder--Narasimhan types which sum to $d$. Under the above isomorphism, the natural surjection $\smash{A_*^{G_d}}(R_d) \to \smash{A_*^{G_d}}(R_d^{\sst})$ is the projection to the summand $\smash{A_*^{G_d}}(\smash{R_{(d)}^{\HN}}) = \smash{A_*^{G_d}}(R_d^{\sst})$ and the push-forward of the closed embedding of the unstable locus corresponds to the embedding of the direct summand $\smash{\bigoplus_{d^* \neq (d)}} \smash{A_*^{G_d}}(\smash{R_{d^*}^{\HN}})$. In particular, both maps are split epi-/monomorphisms.
	Using the cohomological grading of the Chow groups, we get
	$$
		\smash{A_{G_d}^*}(R_d) \cong \bigoplus_{d^*} \smash{A_{G_d}^{*-\codim_{R_d}(R_{d^*}^{\HN})}}(R_{d^*}^{\HN}).
	$$
	The codimension of $R_{d^*}^{\HN}$ in $R_d$ can easily be computed as $\chi(d^*) := \sum_{r < s} \chi(d^r,d^s)$. With the arguments from the proof of Theorem \ref{cyc}, we obtain an isomorphism 
	$$
		A_{G_d}^*(R_{d^*}^{\HN}) \cong A_{G_{d^1}}^*(R_{d^1}^{\sst}) \otimes \ldots \otimes A_{G_{d^l}}^*(R_{d^l}^{\sst}).
	$$
	From the sections of the pull-backs of the open embeddings of the individual semi-stable loci, we obtain a section of the natural surjection
	$$
		A_{G_{d^1}}^*(R_{d^1}) \otimes \ldots \otimes A_{G_{d^l}}^*(R_{d^l}) \onto A_{G_{d^1}}^*(R_{d^1}^{\sst}) \otimes \ldots \otimes A_{G_{d^l}}^*(R_{d^l}^{\sst})
	$$
	and this section gives rise to a commutative diagram
	\begin{center}
		\begin{tikzpicture}[description/.style={fill=white,inner sep=2pt}]
			\matrix(m)[matrix of math nodes, row sep=1.5em, column sep=3em, text height=1.5ex, text depth=0.25ex]
			{
				A_{G_{d^1}}^*(R_{d^1}^{\sst}) \otimes \ldots \otimes A_{G_{d^l}}^*(R_{d^l}^{\sst}) & A_{G_d}^*(R_d^{\HN}) \\
				A_{G_{d^1}}^*(R_{d^1}) \otimes \ldots \otimes A_{G_{d^l}}^*(R_{d^l}) & A_{G_d}^{*+\chi(d^*)}(R_d) \\
			};
			\path[->, font=\scriptsize]
			(m-1-1) edge node[auto] {$\cong$} 	(m-1-2)
			(m-1-1) edge 	(m-2-1)
			(m-1-2) edge 	(m-2-2)
			(m-2-1) edge 	(m-2-2);
		\end{tikzpicture}
	\end{center}
	in which the lower horizontal map is the ChowHa multiplication, the left vertical map is the section, and the right vertical map is the inclusion as a direct summand. We define the descending tensor product $\bigotimes^{\ot}_{\mu \in \Q} \AA^{\sst,\mu}$ as the $\Gamma$-graded vector space
	$$
		\bigoplus_d \bigoplus_{d^*} A_{G_{d^1}}^*(R_{d^1}^{\sst}) \otimes \ldots \otimes A_{G_{d^l}}^*(R_{d^l}^{\sst})
	$$
	where the inner sum ranges over all Harder--Narasimhan types $d^*$ summing to $d$ (i.e.\ tuples $d^* = (d^1,\ldots,d^l)$ of dimension vectors of slopes $\mu^1 > \ldots > \mu^l$ such that $d^1 + \ldots +d^l = d$). The above considerations then prove
	
	\begin{thm}\label{t:tensor}
		The ChowHa-multiplication induces an isomorphism
		$$
			\bigotimes^{\ot}_{\mu \in \Q} \AA^{\theta-\sst,\mu} \xto{}{\cong} \AA
		$$
		of $\Gamma$-graded vector spaces between the descending tensor product of the $\theta$-semi-stable ChowHa's over all possible slopes and the ChowHa.
	\end{thm}
	
	\begin{rem}
		The theorem is valid with integral coefficients, for an arbitrary quiver, and does not require the stability condition to be generic. Theorem \ref{t:tensor} has been proved with different methods by Rim{\'a}nyi in \cite{Rimanyi:13} for the CoHa of a Dynkin quiver which is not an orientation of $E_8$.
	\end{rem}

	\section{Structure of the CoHa of a Symmetric Quiver}\label{s:structurecoha}
	
	The CoHa/ChowHa of a quiver is described explicitly in \cite{KS:11}. Since we will make use of this description, we recall it here: The equivariant Chow ring $A_{G_d}^*(R_d) \cong A_{G_d}^*(\pt)$ is isomorphic to
	$$
		\Q[x_{i,r} \mid i \in Q_0,\ 1 \leq r \leq d_i]^{W_d},
	$$
	where $W_d = \prod_i S_{d_i}$ is the Weyl group of a maximal torus of $G_d$. We may regard the variables $x_{i,r}$ (located in degree $1$) as a basis for the character group of this torus or as the Chern roots of the $G_d$-linear vector bundle $R_d \times k^{d_i} \to R_d$ with $G_d$ acting on $k^{d_i}$ by its $i$\textsuperscript{th} factor.
	
	\begin{thm}[{\cite[Thm.\ 2]{KS:11}}]
		For $f \in \AA_d$ and $g \in \AA_e$, the product $f * g$ equals the function
		$$
			\sum f(x_{i,\sigma_i(r)} \mid i,\ 1 \leq r \leq d_i) \cdot g(x_{i,\sigma_i(d_i+s)} \mid i,\ 1 \leq s \leq e_i) \cdot \prod_{i,j \in Q_0} \prod_{r=1}^{d_i} \prod_{s=1}^{e_j} (x_{j,\sigma_j(d_j+s)} - x_{i,\sigma_i(r)})^{a_{i,j} - \delta_{i,j}}.
		$$
		The sum ranges over all $(d,e)$-shuffles $\sigma = (\sigma_i \mid i) \in W_{d+e}$, that means each $\sigma_i$ is a $(d_i,e_i)$-shuffle permutation.
	\end{thm}
	
	We assume that the stability condition $\theta$ is $\mu$-generic. In this case, we can equip the semi-stable ChowHa of slope $\mu$ with a refined grading: setting
	$$
		\AA_{(d,n)}^{\sst} = \begin{cases} A_{G_d}^{\frac{1}{2}(n-\chi(d,d))}(R_d^{\sst}), & n \equiv \chi(d,d) \text{ (mod 2)} \\ 0, & n \not\equiv \chi(d,d) \text{ (mod 2),} \end{cases}
	$$
	it is easy to see that the multiplication map becomes bigraded, thus $\AA_{(d,n)}^{\sst} \otimes \AA_{(e,m)}^{\sst} \to \AA_{(d+e,n+m)}^{\sst}$.
	
	Like in Section \ref{DT}, we consider again the case of a symmetric quiver and the trivial stability condition. In this situation, it is immediate from the formula in the above theorem that $f * g = (-1)^{\chi(d,e)} g * f$ for $f \in \AA_d$ and $g \in \AA_e$. One can show (cf.\ \cite[Sect.\ 2.6]{KS:11}) that there exists a bilinear form $\psi$ on the $\Z/2\Z$-vector space $(\Z/2\Z)^{Q_0}$ such that $f \star g = (-1)^{\psi(d,e)} f*g$ is a super-commutative multiplication, when defining the parity of an element of bidegree $(d,n)$ to be the parity of $n$.
	We see that the generating series $P(q,t) = \sum_d \sum_k (-1)^k \dim \AA_{(d,k)} q^{k/2} t^d$ is
	$$
		\sum_d (-q^{1/2})^{\chi(d,d)} \prod_i \prod_{\nu=1}^{d_i}(1-q^\nu)^{-1}t^d.
	$$
	So, $P(q,t) = A(q^{-1},t)$. By Theorem \ref{dt_thm}, the generating series has a product expansion
	$$
		P(q,t) = \prod_d \prod_k \prod_{n \geq 0} (1-q^{n+k/2}t^d)^{(-1)^{k-1}\Omega_{d,k}}.
	$$
	As a free super-commutative algebra with a generator in bidegree $(d,k)$ has the generating series $\smash{(1-q^{k/2}t^d)^{(-1)^{k-1}}} = \Exp((-1)^kq^{k/2}t^d)$, Kontsevich--Soibelman made a conjecture in \cite{KS:11} which was eventually proved by Efimov.
	
	\begin{thm}[{\cite[Thm.\ 1.1]{Efimov:12}}] \label{Efimov}
		For a symmetric quiver $Q$, the algebra $\AA(Q)$, equipped with the super-commutative multiplication $\star$, is isomorphic to a free super-commutative algebra over a $(\Gamma \times \Z)$-graded vector space $V = V^{\prim} \otimes \Q[z]$, where $z$ lives in bidegree $(0,2)$, and $\bigoplus_k V_{d,k}^{\prim}$ is finite-dimensional for every $d$.
	\end{thm}
	
	This result implies that the DT invariants $\Omega_{d,k}$ must agree with the dimension of $V_{(d,k)}^{\prim}$ and must therefore be non-negative. We will give another characterization of the primitive part of the CoHa in Theorem \ref{DT=ChowBetti}.
	
	
	\section{Tautological Presentation of the Semi-Stable ChowHa}\label{s:taut}
	
	We investigate the relation between the semi-stable ChowHa $\AA^{\theta-\sst,\mu}$ and the ChowHa $\AA$ of a quiver $Q$. For a dimension vector $d$ of slope $\mu$, we consider the open embedding
	$$
		R_d^{\sst} \to R_d
	$$
	which gives rise to a surjective map $A_{G_d}^*(R_d) \to A_{G_d}^*(R_d^{\sst})$. As the Hecke correspondences for the semi-stable ChowHa are given by restricting the Hecke correspondences of $\AA$ to the semi-stable loci, these open pull-backs are compatible with the multiplication, i.e.\ they induce a surjective homomorphism of $\Gamma$-graded algebras
	$$
		\AA \to \AA^{\theta-\sst,\mu}.
	$$
	Here, we regard $\AA^{\theta-\sst,\mu}$ as a $\Gamma$-graded algebra by extending it trivially to every dimension vector whose slope is not $s$. We can describe the kernel explicitly:
	
	\begin{thm} \label{taut}
		The kernel of the natural map $\AA_d \to \AA_d^{\theta-\sst}$ equals the sum
		$$
			\sum \AA_p * \AA_q
		$$
		over all pairs $(p,q)$ of dimension vectors of $Q$ which sum to $d$ and such that $\mu(p) > \mu(q)$.
	\end{thm}
	
	The key ingredient of the proof of this result is a purely intersection-theoretic lemma. Following \cite[B.1.1]{Fulton:98}, we call a $k$-scheme algebraic if it is separated and of finite type over $\Spec k$. Thus, a variety is an algebraic scheme which is integral.
	
	\begin{lem} \label{push}
		Let $f: X \to Y$ be a surjective, proper morphism of algebraic $k$-schemes. Then the push-forward $f_*: A_*(X)_\Q \to A_*(Y)_\Q$ is surjective.
	\end{lem}
	
	\begin{proof}
		It is obviously sufficient to prove that, for every dominant morphism $f: X \to Y$ of an algebraic scheme $X$ to a variety $Y$, there exists a subvariety $W$ of $X$ of dimension $\dim W = \dim Y$ which dominates $Y$. This is a local statement, so we may assume $X$ and $Y$ to be affine, say $X = \Spec B$ and $Y = \Spec A$. The morphism $f$ corresponds to an extension $A \into B$ of rings. We therefore need to show that there exists a prime ideal $\qq$ of $B$ with $\qq \cap A = (0)$ such that the induced extension
		$$
			Q(B/\qq) \mid Q(A)
		$$
		is finite. Let $K = Q(A)$ and $R = B \otimes_A K$. By Noether Normalization, there exist $b_1,\ldots,b_n \in R$, algebraically independent over $K$, such that $K[b_1,\ldots,b_n] \sub R$ is a finite (and hence integral) ring-extension. Without loss of generality, we may assume $b_1,\ldots,b_n \in B$. Choose a set of generators $c_1,\ldots,c_s$ of $R$ as a $K[b_1,\ldots,b_n]$-algebra and polynomials $p_i(T) \in K[b_1,\ldots,b_n][T]$ such that $p_i(c_i) = 0$. We find an element $s \in A - \{0\}$ such that the coefficients of all the $p_i$'s lie in $A_s[b_1,\ldots,b_n]$ and $p_i(c_i) = 0$ holds in $B_s$. This implies that $B_s$ is an integral $A_s[b_1,\ldots,b_n]$-algebra which yields the surjectivity of the map $\Spec B_s \to \Spec A_s[b_1,\ldots,b_n]$. We consider the prime ideal $\pp' = (b_1,\ldots,b_n)$ of $A_s[b_1,\ldots,b_n]$ and find a prime ideal $\qq'$ of $B_s$ which lies above it. Then $B_s/\qq'$ is an integral extension of $A_s[b_1,\ldots,b_n]/\pp' = A_s$ and therefore, setting $\qq = \qq' \cap B$, the extension
		$$
			Q(B/\qq) = Q(B_s/\qq') \mid Q(A_s) = Q(A)
		$$
		is finite.
	\end{proof}
	
	\begin{proof}[\normalfont \textit{Proof of Theorem \ref{taut}}]
		Let $R_d^{\unst}$ be the complement of $R_d^{\sst}$ in $R_d$. Then, we have an exact sequence
		$$
			A_m^G(R_d^{\unst}) \to A_m^G(R_d) \to A_m^G(R_d^{\sst}) \to 0,
		$$
		where $G = G_d$. For a decomposition $d = p+q$, let $R_{p,q}$ be the closed subset of $R_d$ of all representations which possess a subrepresentation of dimension vector $p$. It is the $G$-saturation of $\smash{\uppertripq}$. The $G$-action gives a surjective, proper morphism
		$$
			\uppertripq \times^{P_{p,q}} G \to R_{p,q},
		$$
		where $P_{p,q}$ is the parabolic $\big( \smash{\begin{smallmatrix} G_p & * \\ & G_q \end{smallmatrix}} \big)$. The unstable locus $R_d^{\unst}$ equals the union $\bigcup R_{p,q}$ over all decompositions $d = p+q$ where the slope of $p$ is larger than the slope of $q$. Let us call these decompositions $\theta$-forbidden. We obtain, using Lemma \ref{push} and \cite[Ex.\ 1.3.1 (c)]{Fulton:98}, that the sequence
		$$
			\bigoplus A_m^G\left( \uppertripq \times^{P_{p,q}} G \right) \to A_m^G(R_d) \to A_m^G(R_d^{\sst}) \to 0
		$$
		is exact when passing to rational coefficients --- the direct sum being taken over all forbidden decompositions $d=p+q$. Setting $n = \dim R_d - m$, we identify 
		$$
			A_m^G\left( \uppertripq \times^{P_{p,q}} G \right) \cong A_{P_{p,q}}^{n+\chi(p,q)} (\uppertripq) \cong A_{G_p \times G_q}^{n+\chi(p,q)}(R_p \times R_q)
		$$ 
		like in Section \ref{s:sstchowha}. As the equivariant product map $\smash{A_{G_p}^*}(R_p) \otimes \smash{A_{G_q}^*}(R_q) \to \smash{A_{G_p \times G_q}^*}(R_p \times R_q)$ is an isomorphism (which is clear from the explicit description given above), we have shown that
		$$
			\bigoplus_{\substack{p+q=d \\ \text{forbidden}}} \bigoplus_{k+l=n+\chi(p,q)} A_{G_p}^k(R_p)_\Q \otimes A_{G_q}^l(R_q)_\Q \to A_{G_d}^n(R_d)_\Q \to A_{G_d}^n(R_d^{\sst})_\Q \to 0
		$$
		is an exact sequence. The first map in this sequence is precisely the ChowHa-multiplication. This proves the theorem.
	\end{proof}

	\section{The Primitive Part of the Semi-Stable ChowHa}\label{s:primitive}
	
	As a next step, we analyze the kernel of the pull-back $A_G^*(R_d^{\sst}) \to A_G^*(R_d^{\st})$ induced by the open embedding of the stable locus into the semi-stable locus. A semi-stable representation $M \in R_d$ is not stable if and only if there exists a proper subrepresentation of the same slope. For a decomposition $d = p+q$ into sub-dimension vectors of the same slope, we define $\smash{R_{p,q}^{\sst}}$ as the subset of those $M \in \smash{R_d^{\sst}}$ which admit a subrepresentation of dimension vector $p$. Therefore, the set of properly semi-stable representations is the union
	$$
		R_d^{\sst} - R_d^{\st} = \bigcup R_{p,q}^{\sst}
	$$
	over all decompositions $d = p+q$ such that $p$ and $q$ have the same slope and are both non-zero. We have, yet again, a surjective, proper morphism 
	$$
		\big( \smash{ \begin{smallmatrix} R_p^{\sst} & * \\ & R_q^{\sst} \end{smallmatrix} } \big) \to R_{p,q}^{\sst}.
	$$
	A result of Totaro (see \cite[Le.\ 6.1]{Totaro:99}), which can easily be transferred to equivariant Chow rings, shows that the exterior product $\smash{A_{G_p}^*}(R_p^{\sst}) \otimes \smash{A_{G_q}^*}(R_q^{\sst}) \to \smash{A_{G_p \times G_q}^*}(R_p^{\sst} \times R_q^{\sst})$ is an isomorphism. Following the arguments of the proof of Theorem \ref{taut}, we obtain:
	
	\begin{thm} \label{prim}
		The kernel of the surjection $\AA_d^{\theta-\sst} \to \AA_d^{\theta-\st}$ is the sum
		$$
			\sum \AA_p^{\theta-\sst} * \AA_q^{\theta-\sst} 
		$$
		over all decompositions $d = p+q$ into non-zero sub-dimension vectors of the same $\theta$-slope.
	\end{thm}
	
	In other words, the graded vector space $\AA^{\theta-\st,\mu} = \bigoplus_{d \in \Gamma^{\theta,\mu}} \AA_d^{\theta-\st}$ equipped with the trivial multiplication (by which we mean that the product of two homogeneous elements of positive degree is set to be zero) is isomorphic to the quotient $\AA^{\theta-\sst,\mu}/(\AA_+^{\theta-\sst,\mu} * \AA_+^{\theta-\sst,\mu})$ of the semi-stable ChowHa modulo the square of its augmentation ideal $\smash{\AA_+^{\theta-\sst}}$.
	
	Again, we consider the case of a symmetric quiver $Q$. We have deduced from Theorem \ref{taut} that $\AA^{\theta-\sst,\mu}$ is free super-commutative over $V^{\theta,\mu} = \bigoplus_{d \in \Gamma^{\theta,\mu}} V_d$. The quotient of the augmentation ideal of a free super-commutative algebra by its square is isomorphic to the primitive part of the algebra, i.e.\ in our case
	$$
		V_d \cong \AA_d^{\theta-\st} = A_{G_d}^*(R_d^{\theta-\st}) \cong A^*_{PG_d}(R_d^{\theta-\st}) \otimes A_{\G_m}^*(\pt)
	$$
	for every $d \neq 0$. As $V_d = V_d^{\mathrm{prim}} \otimes \Q[z]$, we deduce that 
	$$
		V_{d,k}^{\prim} = \begin{cases} A_{PG_d}^{\frac{1}{2}(k-\chi(d,d))}(R_d^{\st}), & k \equiv \chi(d,d) \text{ (mod 2)} \\ 0, & k \not\equiv \chi(d,d) \text{ (mod 2).} \end{cases}
	$$
	Assuming that $R_d^{\theta-\st}$ is non-empty and denoting by $M_d^{\theta-\st}$ the geometric quotient $R_d^{\theta-\st}/PG_d$ (which we call the stable moduli space), we get 
	$$
		A_{PG_d}^j(R_d^{\theta-\st}) = A_{\dim R_d-j}^{PG_d}(R_d^{\theta-\st}) = A_{\dim R_d - \dim PG_d -j}(M_d^{\theta-\st})
	$$
	and $\dim M_d^{\theta-\st} = \dim R_d - \dim PG_d = 1 - \chi(d,d)$. This yields that the Donaldson--Thomas invariants of $Q$ are given by the Chow--Betti numbers of the stable moduli spaces, more precisely:
	
	\begin{thm} \label{DT=ChowBetti}
		For a symmetric quiver $Q$, a stability condition $\theta$ and a dimension vector $d \neq 0$, the Donaldson--Thomas invariant $\Omega_{d,k}$ equals
		$$
			\Omega_{d,k} = \begin{cases} \dim A_{1 - \frac{1}{2}(k+\chi(d,d))}(M_d^{\st}), & \text{if } k \equiv \chi(d,d) \text{ (mod 2) and } M_d^{\st} \neq \emptyset \text{ and}\\ 0, & \text{otherwise.} \end{cases}
		$$
		In particular, $\Omega_{d,k}$ can only be non-zero if $\chi(d,d) \leq k \leq 2 - \chi(d,d)$.
	\end{thm}
	
	\begin{rem}
		The range for the non-vanishing of the DT invariants from the above theorem yields that the number $N_d(Q)$ in \cite[Cor.\ 4.1]{Efimov:12} can be chosen as $1 - \chi(d,d)$, i.e.\ the dimension of $M_d^{\theta-\st}$.
	\end{rem}

	\section{Examples} \label{s:examples}
	
	\subsection{The Two-Cycle Quiver}
	
	We start by illustrating the tensor product decomposition Corollary \ref{t:tensor}. There are exactly three connected symmetric quivers which are not wild, that is, for which a classification of their finite-dimensional representations up to isomorphism is known. Namely, these are
	\begin{itemize}
	\item the quiver $L_0$ of Dynkin type $A_1$ with a single vertex and no arrows,
	\item	the quiver $L_1$ of extended Dynkin type $\widetilde{A}_0$ consisting of a single vertex and a single loop,
	\item and the quiver $Q$ of extended Dynkin type $\widetilde{A}_1$ with two vertices $i$ and $j$ and single arrows $i\rightarrow j$ resp.~$j\rightarrow i$.
	\end{itemize}
	
	For the quivers $L_0$ and $L_1$, the structure of the CoHa is determined in \cite{KS:11}. Namely, we have
	$$\AA(L_0)\cong S^*({\Q}(1,1)[z])\mbox{ and }\AA(L_1)\cong S^*({\Q}(1,0)[z]),$$
	where $S^*$ denotes the free super-commutative algebra, ${\Q}(d,i)$ denotes a one-dimensional ${\Q}$-space placed in bidegree $(d,i)$, and $z$ denotes the element in bidegree $(0,2)$ whose existence is guaranteed by Theorem \ref{Efimov}.
	
	The structure of the CoHa of $Q$ is described in \cite[Cor.\ 2.5]{Franzen:15:CoHa_Modules}; here we give a simplified derivation of this result using the present methods. We consider the stability $\theta$ given by $\theta(d_i,d_j)=d_i$ (note that any non-trivial stability is equivalent to $\theta$ or $-\theta$ in the sense that the class of (semi\nobreakdash-)stable representation is the same). Let a representation $M$ of $Q$ of dimension vector $d$ be given by vector spaces $V_i$ and $V_j$ and linear maps $f:V_i\rightarrow V_j$, $g:V_j\rightarrow V_i$. We claim that this representation is $\theta$-semi-stable if and only if $V_i=0$, or $V_j=0$, or $\dim V_i=\dim V_j$ and $f$ is an isomorphism; moreover, it is $\theta$-stable if it is $\theta$-semi-stable and $\dim V_i$, $\dim V_j\leq 1$.
	
	The case $(\dim V_i)\cdot(\dim V_j)=0$ being trivial, we assume $\dim V_i$, $\dim V_j\geq 1$. Suppose $M$ is $\theta$-semi-stable. If $f$ is not injective, we choose a vector $0\not=v\in V_i$ in the kernel of $f$, yielding a subrepresentation $U$ of dimension vector $(1,0)$. Then we find $1=\mu(U)\leq\mu(M)=\dim V_i/(\dim V_i+\dim V_j)$, thus $\dim V_j=0$, a contradiction. Thus $f$ is injective, and $(V_i,f(V_i))$ defines a subrepresentation $U'$ of dimension vector $(\dim V_i,\dim V_i)$ of $M$. Then we find $1/2=\mu(U')\leq \mu(M)$, thus $\dim V_j\leq\dim V_i$, which already implies $\dim V_i=\dim V_j$ and shows that $f$ is an isomorphism. In this case, $M$ is $\theta$-semi-stable since the subrepresentations are of the form $(U,f(U))$, and thus of the same slope as $M$, for $U\subset V_i$ a $gf$-stable subspace. Since such subspaces always exist, we also see that stability forces $\dim V_i=1=\dim V_j$.
	
	This analysis provides identifications
	\begin{align*}
		A_{G_d}^*(R_d^{\theta-\sst}) &\cong A_{\Gl_n(k)}^*(\pt)\mbox{ for }d=(n,0)\mbox{ or }d=(0,n),\\
		A_{G_d}^*(R_d^{\theta-\sst}) &\cong A_{\Gl_n(k)}^*(M_{n\times n}(k))\mbox{ for }d=(n,n),
	\end{align*}
	which we recognize as the homogeneous parts of the CoHa of $L_0$ and $L_1$, respectively. These identifications obviously being compatible with the respective Hecke correspondences defining the multiplications, we see that
	\begin{align*}
		\AA^{\theta-{\sst},1}(Q)\cong \AA(L_0) &\cong \AA^{\theta-{\sst},0}(Q)\mbox{ and} \\
		\AA^{\theta-{\sst},1/2}(Q) &\cong \AA(L_1).
	\end{align*}
	By Corollary \ref{t:tensor}, we thus arrive at
	$$\AA(Q)\cong S^*\big( (\Q((1,0),1)\oplus\Q((0,1),1)\oplus\Q((1,1),0))[z] \big).$$
	
	\subsection{The Kronecker Quiver}
	
	Now we consider the Kronecker quiver $K_2$ with two vertices $i$ and $j$ and two arrows from $i$ to $j$. As we will use results from Section \ref{s:chowhavscoha}, we work over the field of complex numbers. Again we consider the stability $\theta(d_i,d_j)=d_i$. This is again a case where the representation theory of the quiver is known: up to isomorphism, there exist unique ($\theta$-stable) indecomposable representations $P_n$ resp.~$I_n$ for each of the dimension vectors $(n,n+1)$ resp.~$(n+1,n)$ for $n\geq 0$, and there exist one-parametric families $R_n(\lambda)$ of ($\theta$-semi-stable) indecomposables for each of the dimension vectors $(n,n)$ for $n\geq 0$ and $\lambda\in\P^1(\C)$. Arguing as in the first example, we can conclude that
	$$\AA^{\theta-{\sst},\mu(d)}(K_2)\cong S^*(\Q(d,1)[z])\mbox{ for }d=(n,n+1)\mbox{ or }d=(n+1,n),$$
	and
	$$\AA^{\theta-{\sst},\mu}(K_2)=0\mbox{ if }\mu\not\in\left\{0,\frac{1}{3},\frac{2}{5},\frac{3}{7},\ldots,\frac{1}{2},\ldots\frac{4}{7},\frac{3}{5},\frac{2}{3},1 \right\}.$$
	It remains to consider $\AA^{\theta-{\sst},1/2}(K_2)$.
	
	We construct a stratification of the $\theta$-semi-stable locus in $R_{(n,n)}(K_2)\cong M_{n\times n}(\C)\times M_{n\times n}(\C)$, on which $G=\Gl_n(\C)\times\Gl_n(\C)$ acts via $(g,h)\cdot(A,B)=(hAg^{-1},hBg^{-1})$. For $0\leq r\leq n$, we define $S_r$ as the $G$-saturation of the set of pairs of matrices
	$$\left( \left(\begin{array}{cc}E_r&0\\ 0&N\end{array}\right),\left(\begin{array}{cc}A&0\\ 0&E_{n-r}\end{array}\right) \right),$$
	where $E_i$ denotes an $i\times i$-identity matrix, $A$ denotes an arbitrary $r\times r$-matrix, and $N$ denotes a nilpotent $(n-r)\times(n-r)$-matrix. We claim that every $S_r$ is locally closed, their union equals the $\theta$-semi-stable locus, and the closure of $S_r$ equals the union of the $S_{r'}$ for $r'\leq r$.
	
	The representation $R_n(\lambda)$ is given explicitly by the matrices $(E_n,\lambda E_n+J_n)$ for $\lambda\not=\infty$, and by $(J_n,E_n)$ for $\lambda=\infty$, where $J_n$ is the nilpotent $n\times n$-Jordan block. As noted above, a $\theta$-semi-stable representation of $M$ of dimension vector $(n,n)$ is of the form
	$$M=R_{n_1}(\lambda_1)\oplus\ldots\oplus R_{n_k}(\lambda_k)$$
	for $n=n_1+\ldots+n_k$ and $\lambda_1,\ldots,\lambda_k\in \P^1(\C)$, uniquely defined up to reordering. Now we reorder the direct sum and assume that $\lambda_1,\ldots,\lambda_j\not=\infty$ and $\lambda_{j+1}=\ldots=\lambda_k=\infty$. Using the above explicit form of the representations $R_n(\lambda)$, we see that $M$ is represented by a pair of block matrices of the form
	$$\left( \left(\begin{array}{cc}E_r&0\\ 0&N\end{array}\right),\left(\begin{array}{cc}A&0\\ 0&E_{n-r}\end{array}\right) \right)$$
	with $N$ nilpotent and $A$ arbitrary. All claimed properties of the stratification follow.

	Now we claim that
	$$S_r\cong(\Gl_n(\C)\times \Gl_n(\C))\times^{\Gl_r(\C)\times \Gl_{n-r}(\C)}(M_r(\C)\times N_{n-r}(\C)),$$
	where the group $\Gl_r(\C) \times \Gl_{n-r}(\C)$ is considered as a subgroup of $\Gl_n(\C) \times \Gl_n(\C)$ by mapping a pair $(g_1,h_4)$ to $\big( \big(\begin{smallmatrix} g_1 & 0 \\ 0 & h_4 \end{smallmatrix}\big), \big(\begin{smallmatrix} g_1 & 0 \\ 0 & h_4 \end{smallmatrix}\big) \big)$.
	We consider the stabilizer of the set of matrices in the above block form. So we take $g,h\in\Gl_n(\C)$, written as block matrices
	$$g=\left(\begin{array}{cc}g_1&g_2\\ g_3&g_4\end{array}\right),\;
	h=\left(\begin{array}{cc}h_1&h_2\\ h_3&h_4\end{array}\right),$$
	and assume we are given matrices $A,A'\in M_{r\times r}(\C)$ and $N,N'\in N_{n-r}(\C)$, the nilpotent cone of $(n-r)\times(n-r)$-matrices, such that
	$$\left(\begin{array}{cc}h_1&h_2\\ h_3&h_4\end{array}\right)\left(\begin{array}{cc}E_r&0\\ 0&N\end{array}\right)=\left(\begin{array}{cc}E_r&0\\ 0&N'\end{array}\right)\left(\begin{array}{cc}g_1&g_2\\ g_3&g_4\end{array}\right),$$
	$$\left(\begin{array}{cc}h_1&h_2\\ h_3&h_4\end{array}\right)\left(\begin{array}{cc}A&0\\ 0&E_{n-r}\end{array}\right)=\left(\begin{array}{cc}A'&0\\ 0&E_{n-r}\end{array}\right)\left(\begin{array}{cc}g_1&g_2\\ g_3&g_4\end{array}\right).$$
	From these equations we first conclude $h_1=g_1$ and $h_4=g_4$, thus $h_2=A'g_2$ and $g_2=h_2N$, which yields $h_2=A'h_2N$. By induction, this implies $h_2=(A')^kh_2N^k$ for all $k\geq 1$. But $N$ is nilpotent, thus $h_2=0$, thus $g_2=0$. Similarly, we can conclude $h_3=0$ and $g_3=0$. But then $g_1$ and $g_4$ are invertible, and $A'=g_1Ag_1^{-1}$ as well as $N'=g_4Ng_4^{-1}$. This proves the claim.

	To obtain information on the Chow groups from this stratification using Lemma \ref{filt}, we first have to analyze the Chow groups of  nilpotent cones. 

	The nilpotent cone $N_d(\C)$ is irreducible of dimension $d^2-d$, and the $\Gl_d(\C)$-orbits $\mathcal{O}_\lambda$ in $N_d$ are parametrized by partitions $\lambda$ in $\mathcal{P}_d$, the set of partitions of $d$ (we denote by $\mathcal{P}$ the union of all $\mathcal{P}_d$'s). The stabilizer $G_\lambda$ of a point in $\mathcal{O}_\lambda$ has dimension $\langle\lambda,\lambda\rangle=\sum_{i,j}\min(m_i,m_j)m_im_j$, and its reductive part is isomorphic to $\prod_i\Gl_{m_i}(\C)$, where $m_i=m_i(\lambda)$ denotes the multiplicity of $i$ as a part of $\lambda$, for $i\geq 1$. We can thus apply Lemma \ref{filt} and reduce the structure group --- note that in characteristic zero, an orbit is isomorphic to the quotient of the group by the stabilizer of a point --- to get
	$$
		A^*_{\Gl_d(\C)}(N_d(\C))\cong \bigoplus_{\lambda\in\mathcal{P}_d}A^{*+d-\langle\lambda,\lambda\rangle}_{G_\lambda}(\pt),
	$$
	and the equivariant cycle map for $N_d(\C)$ is an isomorphism.

	This enables us to again apply Lemma \ref{filt}, this time to the stratification $(S_r)_r$. We compute (using $\codim S_r=n-r$):
	\begin{align*}
		A_{\Gl_n(\C)\times\Gl_n(\C)}^*(R_{(n,n)}^{\theta-\sst}(K_2)) &\cong \bigoplus_{r=0}^n A_{\Gl_n(\C)\times\Gl_n(\C)}^{*-n+r}(S_r) \\
		&\cong \bigoplus_{r=0}^n A_{\Gl_r(\C)\times\Gl_{n-r}(\C)}^{*-n+r}(M_r(\C)\times N_{n-r}(\C)) \\
		&\cong \bigoplus_{r=0}^n A_{\Gl_r(\C)}^{*}(M_r(\C))\otimes A_{\Gl_{n-r}(\C)}^{*-n+r}(N_{n-r}(\C)) \\
		&\cong \bigoplus_{r=0}^n A_{\Gl_r(\C)}^{*}(\pt)\otimes\bigoplus_{\lambda\in\mathcal{P}_{n-r}}A_{G_\lambda}^{*-\langle\lambda,\lambda\rangle}(\pt)
	\end{align*}

	Summing over all $n$, we obtain
	$$
		\AA_{(*,2*)}^{\theta-{\sst},1/2}(K_2) \cong \bigoplus_{n\geq 0}A_{\Gl_n(\C)\times\Gl_n(\C)}^*(R_{(n,n)}^{\theta-\sst}(K_2)) \cong
		\bigg( \bigoplus_{r\geq 0}A_{\Gl_r(\C)}^*(\pt) \bigg) \otimes \bigg( \bigoplus_{\lambda\in\mathcal{P}}A_{G_\lambda}^{*-\langle\lambda,\lambda\rangle}(\pt) \bigg).
	$$
	The generating function of the bigraded space $\AA^{\theta-\sst,1/2}(K_2)$ therefore equals
	$$
		\bigg( \sum_{n\geq 0}\frac{t^n}{(1-q)\ldots(1-q^n)} \bigg) \cdot \bigg( \sum_\lambda\frac{q^{-\langle\lambda,\lambda\rangle}t^{|\lambda|}}{\prod_{i\geq 1}((1-q)\ldots (1-q^{m_i}))} \bigg)=\prod_{i\geq 0}\frac{1}{1-q^it}\cdot\prod_{i\geq 1}\frac{1}{1-q^it}
	$$
	by standard identities. We thus arrive at an isomorphism of bigraded $\Q$-spaces
	$$\AA^{\theta-{\sst},1/2}(K_2)\cong S^*((\Q((1,1),0)\oplus\Q((1,1),2))[z]).$$
	
	However, this is not an isomorphism of algebras, since we will now exhibit an example showing that the algebra $\AA^{\theta-{\sst},1/2}(K_2)$ is not super-commutative.

	We use the algebraic description of the CoHa of Section \ref{s:structurecoha} together with Theorem \ref{taut}. We have
	$$\AA_{(m,n)}(K_2)\cong \Q[x_1,\ldots,x_m,y_1,\ldots,y_n]^{S_m\times S_n}$$
	with multiplication given as in Section \ref{s:structurecoha}. By Theorem \ref{taut}, $\AA_{(1,1)}^{\theta-{\sst},1/2}(K_2)$ is the factor of $\AA_{(1,1)}(K_2)\cong\Q[x_1,y_1]$ by the image of the multiplication map $\AA_{(1,0)}(K_2)\otimes\AA_{(0,1)}(K_2)\rightarrow\AA_{(1,1)}(K_2)$, thus 
	$$\AA_{(1,1)}^{\theta-{\sst},1/2}(K_2)\cong\Q[x,y]/(x-y)^2.$$
	Again by Theorem \ref{taut}, $\AA_{(2,2)}^{\theta-{\sst},1/2}(K_2)$ is the factor of $\AA_{(2,2)}(K_2)$ by the image of the multiplication map $$\AA_{(2,1)}(K_2)\otimes\AA_{(0,1)}(K_2)\oplus\AA_{(1,0)}(K_2)\otimes\AA_{(1,2)}(K_2)\rightarrow\AA_{(1,1)}(K_2).$$
	The degree of an element in this image is at least $-\chi((2,1),(0,1))=-\chi((1,0),(1,2))=3$.
	A direct calculation shows that for the elements $1,x,y\in\AA_{(1,1)}^{\theta-{\sst},1/2}(K_2)$, we have in $\AA_{(2,2)}^{\theta-{\sst},1/2}(K_2)$:
	\begin{align*}
		1*1 &= 2, \\
		1*x &= y1+y2, \\
		x*1 &= 2(x1+x2)-(y1+y2), \\
		1*y &= -(x1+x2)+2(y1+y2), \text{ and}\\
		y*1 &= x1+x2.
	\end{align*}
	In particular, the (anti-)commutator of $1$ and $x$ does not vanish.

	In light of the previous description of $\AA^{\theta-{\sst},1/2}(K_2)$, it can be expected that there exists a natural filtration on $\AA^{\theta-{\sst},1/2}(K_2)$ such that the associated graded algebra is isomorphic, as a bigraded algebra, to $S^*\big((\Q((1,1),0)\oplus\Q((1,1),2))[z]\big)$.

	\subsection{DT Invariants as Chow--Betti Numbers}

	Next, we illustrate Theorem \ref{DT=ChowBetti}. We consider the symmetric quiver $Q$ with two vertices $i$ and $j$ and $n\geq 1$ arrows from $i$ to $j$ and from $j$ to $i$, and the dimension vector $d=(1,r)$ for $r\leq n$. To determine the quantized Donaldson--Thomas invariant $\smash{\Omega_{d,k}}$, we use the stability $\theta=(r,-1)$, for which $d$ is coprime. Therefore, $\smash{\Omega_{d,k}}=\smash{\Omega_{d,k}^\theta}$ equals the (suitably shifted) Poincar\'e polynomial of the cohomology of the moduli space $R_d^{\theta-\sst}(Q)/PG_d$, which is isomorphic to a vector bundle over the Grassmannian $\Gr_r(k^n)$ \cite[Sect.\ 6.1]{Reineke:15}. By Theorem \ref{DT=ChowBetti}, we can also compute $\smash{\Omega_{d,k}}$ as the (suitably shifted) Poincar\'e polynomial of the Chow ring of the moduli space $R_d^{0-\st}(Q)/PG_d$. Again by \cite{Reineke:15}, this moduli space is isomorphic to the space $X$ of $n\times n$-matrices of rank $r$. Mapping such a matrix to its image defines a $\Gl_n(k)$-equivariant fibration $X\rightarrow \Gr_r(k^n)$, whose fibre is isomorphic to the space of $r\times n$-matrices of highest rank. The latter being open in an affine space, its Chow ring reduces to $\Q$, thus the Chow ring of $X$ is isomorphic to the Chow ring of $\Gr_r(k^n)$ as expected.

	\subsection{Multiple Loop Quivers}

	Finally, we consider the quiver $L_m$ with a single vertex and $m\geq 2$ loops. The quantized Donaldson--Thomas invariants are computed explicitly in \cite{Reineke:12}.
	
	All stability conditions are equivalent for this quiver. Let $M_d^{\simp}$ be the moduli space of simple (which is the same as stable) representations of $L_m$ of dimension $d$. It is obtained as the geometric quotient $R_d^{\simp}/\PGl_d$. The Chow ring $A^*(M_d^{\simp})_\Q = A_{\PGl_d}^*(R_d^{\simp})_\Q$ (we will always work with rational coefficients in this subsection and therefore neglect it in the notation) is a quotient of the equivariant Chow ring $A_{\PGl_d}^*(R_d) = A_{\PGl_d}^*(\pt)$. The group of characters of a maximal torus of $\Gl_d$ identifies with the free abelian group in letters $x_1,\ldots,x_d$, the natural action of the Weyl group $W = S_d$ being the permutation action. A maximal torus of $\PGl_d$ is given by the quotient of the chosen maximal torus of $\Gl_d$ by the diagonally embedded multiplicative group. The corresponding Weyl group is also $S_d$ and the character group is then the submodule
	$$
		X_d = \mathrm{Sp}_{(d-1,1)} = \{a_1x_1 + \ldots + a_dx_d \mid a_1 + \ldots + a_d = 0 \}.
	$$
	The symmetric algebra $\Sym(X_d)$ over $X_d$ is the subalgebra of $\smash{\Q[x_1,\ldots,x_d]}$ generated by $x_j - x_i$ (with $i < j$) and the equivariant Chow ring $A_{\PGl_d}^*(R_d)$ is therefore $\Sym(X_d)^{S_d}$ which identifies with a subalgebra of $A_{G_d}^*(\pt) = \smash{\Q[x_1,\ldots,x_d]^{S_d}}$. As in the proofs of Theorems \ref{taut} and \ref{prim}, the kernel of $A_{\PGl_d}^*(R_d) \to A_{\PGl_d}^*(\smash{R_d^{\simp}})$ is then given by the image of
	$$
		\bigoplus_{\substack{p+q=d\\p,q > 0}} A_{\PGl_d}^*(Z_{p,q} \times^{P_{p,q}} \PGl_d) \to A_{\PGl_d}^*(R_d),
	$$
	where $P_{p,q}$ is the obvious parabolic subgroup of $\PGl_d$ --- this, by the way, can be done for an arbitrary quiver and for the kernels $A_{PG_d}^*(R_d) \to A_{PG_d}^*(\smash{R_d^{\theta-\sst}})$ and $A_{PG_d}^*(\smash{R_d^{\theta-\sst}}) \to A_{PG_d}^*(\smash{R_d^{\theta-\st}})$. The ring $A_{\PGl_d}^*(Z_{p,q})$ is isomorphic to $\smash{\Sym(X_d)^{S_p \times S_{q}}}$ and the push-forward map 
	$$
		m_{p,q}: A_{\PGl_d}^*(Z_{p,q} \times^{P_{p,q}} \PGl_d) \to A_{\PGl_d}^*(R_d)
	$$ 
	can be described algebraically and looks just like the explicit formula from \cite[Thm.\ 2]{KS:11}, i.e.\ given by a shuffle product with kernel $\prod_{i=1}^p \prod_{j=1}^{q} (x_{p+j} - x_i)^{m-1}$ . The relations in $A_{\PGl_d}^*(R_d)$ which present $A^*(\smash{M_d^{\simp}})$ thus have at least degree $(m-1)(d-1)$. In other words, for every $0 \leq i < (m-1)(d-1)$, we get
	$$
		A^i(M_d^{\simp}) \cong A_{\PGl_d}^i(R_d) = \Sym^i(X_d)^{S_d}.
	$$
	The generating series of $\Sym(X_d)^{S_d}$ is
	$$
		\frac{1}{(1-q^2)\ldots(1-q^d)} = \sum_{i \geq 0} \sharp\{ (k_2,\ldots,k_d) \mid 2k_2 + \ldots + dk_d \} q^i
	$$
	and using Theorem \ref{DT=ChowBetti}, we obtain a description of the first few Donaldson--Thomas invariants:
	
	\begin{propsub}
		For the $m$-loop quiver, the Donaldson--Thomas invariant $\Omega_{d,k}$ for a non-negative integer $d$ and an integer $k$ of the same parity as $(1-m)d^2$ satisfying
		$$
			(1-m)d^2 \leq k < (1-m)(d^2 - 2d + 2)
		$$
		computes as
		$$
			\Omega_{d,k} = \sharp \big\{ (k_2,\ldots,k_d) \mid 2k_2 + \ldots + dk_d = \frac{1}{2}((m-1)d^2 + k) \big\}.
		$$
	\end{propsub}

	We conclude the subsection with a computation of the numbers $\Omega_{2,k}$. The ring $\Sym(X_2)^{S_2}$ is the subalgebra of $\smash{\Q[x_1,x_2]^{S_2}}$ which is generated by $(x_2-x_1)^2$. Abbreviate $\Delta = x_2 - x_1$. As a $\smash{\Sym(X_2)^{S_2}}$-module, $\Sym(X_2)$ is generated by $1$ and $\Delta$. The push-forward map $m_{1,1}: \Sym(X_2) \to \Sym(X_2)^{S_2}$ sends $f(x_1,x_2)$ to
	$$
		\big( f(x_1,x_2)+ (-1)^{m-1} f(x_2,x_1) \big)\Delta^{m-1}
	$$
	and therefore, the image of $m_{1,1}$ is the ideal of $\Sym(X_2)^{S_2} = \Q[\Delta^2]$ which is generated by $\Delta^{2\lfloor m/2 \rfloor}$ (i.e.\ $\Delta^m$ if $m$ is even and $\Delta^{m-1}$ if $m$ is odd). We have shown that
	$$
		\Omega_{2,k} = \begin{cases} 1, & \text{if } k \equiv 0 \text{ (mod $4$) and } 4(1-m) - 2 \leq k \leq 4(\lfloor m/2 \rfloor - m) \text{, and} \\ 0, & \text{otherwise.} \end{cases}
	$$

	\bibliographystyle{abbrv}
	\bibliography{Literature}
\end{document}